\documentclass{journal}
\usepackage{tikz}
\usepackage{color}
\usepackage[all]{xy}
\usepackage{amsmath, amssymb}
\usepackage{amsfonts}
\usepackage{mathrsfs}
\usepackage{amsthm}
\usepackage{bm}

\makeatletter
 \def\th@plain{\upshape}
 \makeatother
\numberwithin{equation}{section}
\newtheorem{theorem}{Theorem}[section]
\newtheorem{prop}[theorem]{Proposition}
\newtheorem{lemma}[theorem]{Lemma}
\newtheorem{cor}[theorem]{Corollary}
\newtheorem{Example}[theorem]{Example}
\newtheorem{definition}[theorem]{Definition}
\newtheorem{rem}[theorem]{Remark}

\title{Quandles over a hyperboloid of one sheet and the longitudinal mapping knot invariant for $SL(2,\mathbb{R})$}
\author{
Kentaro Yonemura
\thanks{
the 2010 MSC: \texttt{57K10, 57K12}, E-mail: \texttt{3MA20009Y@s.kyushu-u.ac.jp}}
}
\date{}
\begin{document}

\maketitle
\begin{abstract}
   This paper aims to consider algebraic structures of quandles defined over a hyperboloid of one sheet and compute the related longitudinal mapping for $SL(2,\mathbb{R})$.
\end{abstract}


\section{Introduction}
A quandle is an algebraic system defined by Joyce \cite{Joyce1982} and Matveev \cite{Matveev1982} independently. Joyce and Matveev was motivated by knot theory and  constructed the almost complete knot invariant called the knot quandles or the fundamental quandles of knots. In this paper, we deal with a smooth quandle, which is a differentiable manifold with a smooth quandle-operation, defined over a hyperboloid of one sheet. See Ishikawa \cite{Ishikawa} and Nosaka \cite{Nosaka2019}
for more details on a smooth quandle.

We deal with two problems in this paper.

First, we see that subquandles of conjugacy quandles are generally different from the quandle composed by the Azcan-Fenn \cite{AzcanFenn1994}, even if both of them have the same topological structure. We consider this issue to show that a fact is a special case: Azcan-Fenn \cite{AzcanFenn1994} defined the spherical quandle using Euclidean inner product. On the other hand, Clark-Saito \cite{ClarkSaito} defined a family of quandles on conjugacy classes of $SU(2)$ and called them spherical quandles. The auther \cite{Yonemura} proved that the two types of spherical quandles are compatible, that is, the spherical quandle defined by Azcan-Fenn is isomorphic to one of the spherical quandle defined by Clark-Saito. 

They defined a knot invariant called a longitudinal mapping and calculate it in the case of $SU(2)$ using the quandle.

Secondly, we determine the value of the longitudinal mapping $\mathcal{L}_G^x(K)$ under limited condition: When $G$ is $SL(2,\mathbb{R})$, $K$ is a $(2,n)$-torus knot, and $x$ is conjugate with 
\[
\begin{pmatrix}
e^{r} & 0\\
0 & e^{-r}
\end{pmatrix}\in SL(2,\mathbb{R}).
\]
The longitudinal mapping is a knot invariant defined by Clark-Saito \cite{ClarkSaito}. Clark and Saito explained that their invariant is a extention of the knot colouring polynomial \cite{Eisermann2007}, which is a generalisation of the quandle cocycle invariant \cite{CJKLS}. Our approach is a little different from Clark-Saito \cite{ClarkSaito}. Clark-Saito \cite{ClarkSaito} presented the elemtnts of $SU(2)$ as unitquaternions and used Python to determine the value of the longitudinal mapping for $SU(2)$. On the other hand, we present the elements of $SL(2,\mathbb{R})$ as matrices and use linear algebra to determine the value of the longitudinal mapping for $SL(2,\mathbb{R})$. Our approach may be a useful example when considering the value of the longitudinal mapping for highter dimentional Lie groups.



This paper is organized as follows. In section \ref{section_preliminaries}, the basic notation and facts on quandles and $SL(2,\mathbb{R})$ are presented. In section \ref{section_quandles_over_hyperboloid}, we discuss on the algebraic structure of quandles over a hyperboloid of one sheet. In section \ref{section_determine_coloring}, we determine $X$-colorings with respect to a diagram of $(2,n)$-torus knots for the case $X$ is a subquandle of conjugacy quandles. In the section \ref{section_determine_S^2_1(r)-coloring}, we apply the argument considered in section \ref{section_determine_coloring}. In the section \ref{section_introduction_to_NSK_1to1correspondence}, we introduce a result of Nosaka \cite{Nosaka2015}. In the section \ref{section_S^2_1(r)-coloring_vs_hyperbolic_rep},we apply the result introduced in section \ref{section_introduction_to_NSK_1to1correspondence}. In section \ref{section_computing_longitudinal_map}, we determine the value of the longitudinal mapping for $SL(2,\mathbb{R})$.

\section{Preliminaries}
\label{section_preliminaries}
We recall definitions and facts using in this paper without proofs.

\subsection{Quandle}
We see the definition of a quandle and some basic facts. See Kamada \cite{Kamadabook} and Nosaka \cite{Nosaka2017book} for more details.
\begin{definition}[Joyce \cite{Joyce1982}, Matveev \cite{Matveev1982}]\label{def_quandle}
A \textit{quandle} is a set $X$ with a binary operation $\triangleright:X\times X\to X$ satisfying the three conditions:

\noindent(Q1) $x\triangleright x = x$ for any $x\in X$.

\noindent(Q2) The map $S_y:X\to X$ defined by $x\mapsto x\triangleright y$ is bijective for any $y\in X$.

\noindent(Q3) $(x\triangleright y)\triangleright z = (x\triangleright z)\triangleright(y\triangleright z)$ for any $x,y,z\in X$.
\end{definition}

A subset $Y$ of quandle $X$ is said to be a \textit{subquandle} if the quandle operation of $X$ is closed in $Y$. A map $f:X\to Y$ between quandles is said to be a \textit{quandle homomorphism} if $f(x\triangleright y)=f(x)\triangleright f(y)$ for any $x,y\in X$. A quandle homomorphism is said to be a \textit{quandle isomorphism} if it is bijective.

We see examples of quandles.

\begin{Example}
Suppose $X$ is a set and $x\triangleright y=x$ for any $x,y\in X$. Then $(X,\triangleright)$ is a quandle called a \textit{trivial quandle}.
\end{Example}

\begin{Example}
Suppose $X=\mathbb{Z}/n\mathbb{Z}$ and $x\triangleright y=2y-x$ for any $x,y\in X$. Then $(X,\triangleright)$ is a quandle called a \textit{dihedral quandle} \cite{Takasaki1943}.
\end{Example}

\begin{Example}
\label{conjugation_quandle}
Let $G$ be a group and $\triangleright$ be a binary operation of $G$ defined by $g\triangleright h=h^{-1}gh$ for any $g,h\in G$. Then $(G,\triangleright)$ is a quandle called \textit{conjugacy quandle}. We denote this quandle $\operatorname{Conj}(G)$.
\end{Example}

\begin{Example}
\label{def_augmented_quandle}
Let $G$ be a group, $X$ a set on which $G$ acts from the right, and  $\kappa:X\to G$ a map satisfying the two conditions:
\begin{enumerate}
    \item $\kappa(x\cdot g)=g^{-1}\kappa(x)g$ for any $x\in X$ and $g\in G$.
    \item $x\cdot \kappa(x)=x$ for any $x\in X$.
\end{enumerate}
Suppose $x\triangleright y=x\cdot\kappa(y)$ for any $x,y\in X$. Then $(X,\triangleright)$ is a quandle called an \textit{augmented quandle} \cite{Joyce1982}. We denote this quandle by $(X,G,\kappa)$ or simply $X$. The quandle $X$ is said to be \textit{faithful} if the map $\kappa$ is injective.
\end{Example}

\begin{Example}
\label{example_presentation_as_augmented_quandle}
Suppose $G$ be a group, $X$ be a subquandle of $\operatorname{Conj}(G)$, and $i_X:X\to G$ be an inclusion map. Then $G$ acts on $X$ via conjugation and $(X,G,i_X)$ is a faithfull augumented quandle. The identity map $\operatorname{id}_X:X\to X$ induces a quandle isomorphism from $X$ to $(X,G,i_X)$.
\end{Example}

\begin{Example}
\label{knot quandle_def}
Let $K$ be a tame knot in the 3-sphere, $\pi_K=\pi_1(S^3\setminus K)$ the knot group of $K$, and $H=\langle\mathfrak{m},\mathfrak{l\rangle}$ the subgroup of $\pi_K$ generated by a meridian $\mathfrak{m}\in\pi_K$ and the preferred longitude $\mathfrak{l}\in \pi_K$ of $K$. Suppose $X=H\backslash\pi_K$ and $Hx\triangleright Hy=H\mathfrak{m}xy^{-1}\mathfrak{m}y$ for $x,y\in\pi_K$. Then, the algebraic system $(X,\triangleright)$ is a quandle and called the \textit{knot quandle} of $K$ or the \textit{fundamental quandle} of $K$ \cite{Joyce1982,Matveev1982}. We denote this quandle by $Q_K$. 
\end{Example}
The knot quandle, we defined in Example \ref{knot quandle_def}, is a complete knot invariant. See \cite{Joyce1982, Matveev1982,Kamadabook, Nosaka2017book} for more details.

A quandle is said to be an \textit{involutory quandle} or a \textit{kei} if
$
(x\triangleright y)\triangleright y=x
$ for any $x,y$. A dihedral quandle is an involutory quandle for example.

We introduce an $X$-coloring and provide some examples.

\begin{definition}[$X$-colorings]
Let $X$ be a quandle, and $D$ be an oriented knot diagram. An $X$\textit{-coloring} of $D$ is a map $C:\{\mbox{arcs of }D\}\to X$ satisfying the condition $C(\alpha_{\tau})\triangleright C(\beta_{\tau})=C(\gamma_{\tau})$ in Fig. \ref{pic_coloring_condition} at each crossing $\tau$ of $D$. We denote the set of $X$-colorings of $D$ by $\operatorname{Col}_X(D)$.
\end{definition}

For a knot $K$ with a diagram $D$ and a quandle $X$, it is known that an $X$-coloring of $D$ induces a quandle homomorphism from $Q_K$ to $X$.

\begin{figure}
  \centering
  \includegraphics[width=7cm]{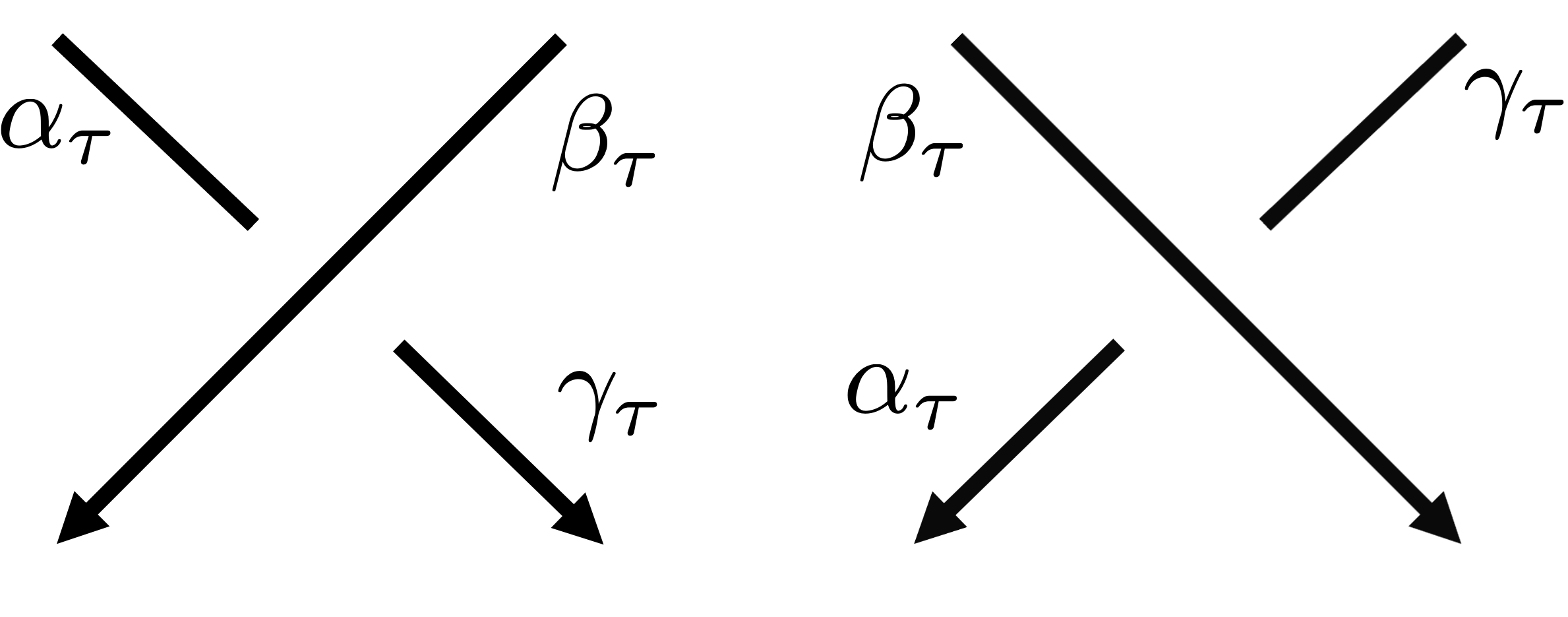}
  \caption{Crossings $\tau$ of a knot diagram $D$}
  \label{pic_coloring_condition}
\end{figure}

We see some special quandle colorings.

\begin{Example}\label{example_trivial-hom}
Suppose $D$ be an oriented knot diagram and $Y=\{*\}$ be a trivial quandle. Then, any map $C:\{\mbox{arcs of }D\}\to Y$ is a $Y$-colorings.
\end{Example}

\begin{definition}\label{def_trivial-hom}
In this paper, an $X$-coloring $C$ is said to be trivial if the image of $C$ is a trivial subquandle of $X$.
\end{definition}

\begin{rem}
Some literature defines a trivial coloring as a coloring  considered in Example \ref{example_trivial-hom}.
\end{rem}

\subsection{Properties of $SL(2,\mathbb{R})$}\label{section_property_SL2R}
Consider the Lie group
\[
SL(2,\mathbb{R})=
\left\{
\begin{pmatrix}
a && b\\
c && d
\end{pmatrix}
\ :\ {}
a,b,c,d\in\mathbb{R},\ ad-bc=1
\right\}.
\]
The Lie group $SL(2,\mathbb{R})$ acts on itself via conjugation. Let $S^2_1(r)$ be a conjugacy class of $SL(2,\mathbb{R})$
\[
S^2_1(r)=
\left\{
g^{-1}
\begin{pmatrix}
e^r && 0\\
0 && e^{-r}
\end{pmatrix}
g
\in{SL}(2,\mathbb{R})
\ :\ {}
g
\in{SL}(2,\mathbb{R})
\right\}
\]
for $r>0$. The conjugacy class $S^2_1(r)$ is diffeomorfic to a hyperboloid of one  sheet. We recall the following known facts.
\begin{prop}\label{prop_S2r_SL2R-orbit}
For any $r>0$, $S^2_1(r)$ is an $SL(2,\mathbb{R})$-orbit.
\end{prop}

\begin{prop}
\label{prop_tr_expS2r}
For any $r>0$,
\[
S^2_1(r)
=\{g\in SL(2,\mathbb{R})\ :\ {}\operatorname{trace}g=2\cosh{r}\}.
\]
\end{prop}

\begin{prop}\label{prop_injectivity_exp}
Let $r$ be positive, and
\[
D(r)=
\begin{pmatrix}
e^r && 0\\
0 && e^{-r}
\end{pmatrix}
\in S^2_1(r).
\] 
Then, both of the isotropy subgroup of $SL(2,\mathbb{R})$ with respect to $D(r)\in SL(2,\mathbb{R})$ is the subgroup consisting of the entire diagonal matrices of $SL(2,\mathbb{R})$.
\end{prop}


\section{Quandles over a hyperboloid of one sheet}
\label{section_quandles_over_hyperboloid}
We consider two types of quandles  $S^2_1(r)$ and ${S^2_1}_{\mathbb{R}}$ over a hyperboloid of one sheet. We see that they are not isomorphic though one of the spherical quandle defined in Clark-Saito \cite{ClarkSaito} , which is similar to $S^2_1(r)$, is isomorphic to the spherical quandle defined by Azcan-Fenn \cite{AzcanFenn1994}, which is similar to ${S^2_1}_{\mathbb{R}}$.

First of all, we see a quandle $S^2_1(r)$ and its property.
\begin{prop}
For $r>0$, the conjugacy class $S^2_1(r)$ is a subquandle of $\operatorname{Conj}(SL(2,\mathbb{R}))$.
\end{prop}
\begin{proof}
It is easy to see the result in light of Proposition \ref{prop_S2r_SL2R-orbit}.
\end{proof}

Secondly, we see the definition of a quandle ${S^2_1}_{\mathbb{R}}$. Let $\langle-,-\rangle:\mathbb{R}^3\times\mathbb{R}^3\to\mathbb{R}$ be a bilinear map defined by
\[
\langle
(x_0,x_1,x_2)
,
(y_0,y_1,y_2)
\rangle
=
-x_0y_0+x_1y_1+x_2y_2,
\]
$S^2_1=\{\bm{x}\in\mathbb{R}^3\ :\ \langle\bm{x},\bm{x}\rangle=1\}$ a hyperboloid of one sheet, and $\triangleright:S^2_1\times S^2_1\to S^2_1$ a binary operation defined by $\bm{x}\triangleright\bm{y}=2\langle\bm{x},\bm{y}\rangle\bm{y}-\bm{x}$ for all $\bm{x},\bm{y}\in S^2_1$.

\begin{prop}[Azcan-Fenn \cite{AzcanFenn1994}]
The algebraic system $(S^2_1,\triangleright)$ is an involutory quandle. We denote this quandle as ${S^2_1}_{\mathbb{R}}$.
\end{prop}

Finally, we see that ${S^2_1}_{\mathbb{R}}$ is different from ${S^2_1}(r)$ for any $r>0$.

\begin{lemma}
\label{lem_S^2_1(r)is_not_inv}
For $r>0$, the quandle $S^2_{1}(r)$ is not an involutory quandle.
\end{lemma}
\begin{proof}
We prove by contradiction. Assume $S^2_1(r)$ be an ivolutory quandle. For $D(r)$, which is defined in Proposition \ref{prop_injectivity_exp}, and any $y\in S^2_1(r)$,
\[
(D(r)\triangleright y)\triangleright y=D(r),
\]
that is, $D(r)y^2=y^2D(r)$. By Proposition \ref{prop_injectivity_exp},
\[
S^2_1(r)
=
\{
D(r)^{\pm 1}
\}.
\]
This is contradiction since $S^2_1(r)$ is not finite.
\end{proof}

\begin{prop}
\label{prop_AzcanFenn_vs_ClarkSaito}
For $r>0$, $S^2_1(r)$ is not isomorphic to ${S^2_1}_{\mathbb{R}}$.
\end{prop}
\begin{proof}
The quandle $S^2_1(r)$ is not an involutory quandle because of Lemma \ref{lem_S^2_1(r)is_not_inv}. On the other hand, the quandle ${S^2_1}_{\mathbb{R}}$ is an involutory quandle. This is a contradiction.
\end{proof}

By Proposition \ref{prop_AzcanFenn_vs_ClarkSaito}, subquandles of conjugacy quandles are generally different from the quandle composed by the Azcan-Fenn \cite{AzcanFenn1994}, even if both of them have the same topological structure.

\section{Quandle Colorings of $(2,n)$-torus knots}
\label{section_determine_coloring}
In this section, we discuss a coloring of $(2,n)$-torus knots with respect to subquandles of conjugacy quadles.

Let $G$ be a group, $X$ be a subquandle of $\operatorname{Conj}(G)$, and $D$ be the diagram of $(2,n)$-torus knots as shown in Fig. \ref{pic_torus_knot_withcaption}. We consider $X$-colorings of the diagram $D$. By the definition of the quandle coloring, \[
\operatorname{Col}_X(D)
=
\left\{
C:
\{\alpha_0,\cdots,\alpha_{n-1}\}
\to
X
:
\begin{array}{rcl}
{}^{\forall}i &=&  0,1,\cdots,n\\
C(\alpha_{i+2}) &=& C(\alpha_{i+1})^{-1}C(\alpha_{i})C(\alpha_{i+1})
\end{array}
\right\}.
\] 

\begin{lemma}
\label{prop_presentation_C_by_0_1}
For $j=0,1,\cdots,\tfrac{n-1}{2}$ and $C\in \operatorname{Col}_{X}(D)$, 
\[
\begin{array}{ccc}
C(\alpha_{2j}) &=&  
\big(C(\alpha_{0})C(\alpha_{1})\big)^{-j}
C(\alpha_0)
\big(C(\alpha_{0})C(\alpha_{1})\big)^{j}\\
C(\alpha_{2j+1})&=& 
(C(\alpha_{0})C(\alpha_{1}))^{-j}
C(\alpha_1)
(C(\alpha_{0})C(\alpha_{1}))^{j}.
\end{array}
\]
\end{lemma}

\begin{proof}
The result follows by induction on $j$.
\end{proof}

For $x,y\in X$, we define a map $C_{x,y}:\{\alpha_0,\cdots,\alpha_{n-1}\}\to X$ as
\[
C_{x,y}(\alpha_{m})=
\left\{
\begin{array}{ccl}
    (xy)^{-j}x(xy)^{j} &\mbox{if}& m=2j  \\
    (xy)^{-j}y(xy)^{j} &\mbox{if}& m=2j+1
\end{array}
\right..
\]
The set $\operatorname{Col}_X(D)$ is presented as follows.
\begin{prop}Suppose $k\geq1$ and $n=2k+1$, 
\label{prop_general_ver_coloring}
\[
\operatorname{Col}_X(D)
=
\left\{
C_{x,y}:\{\alpha_0,\cdots,\alpha_{n-1}\}\to X
\ :\ {}
x,y\in X\mbox{ s.t. }(xy)^{k}x=y(xy)^{k}
\right\}.
\] 
\end{prop}
\begin{proof}
The result follows from Lemma \ref{prop_presentation_C_by_0_1} and $\alpha_{0}=\alpha_{n}=\alpha_{2k+1}$.
\end{proof}

\begin{figure}
  \centering
  \includegraphics[width=7cm]{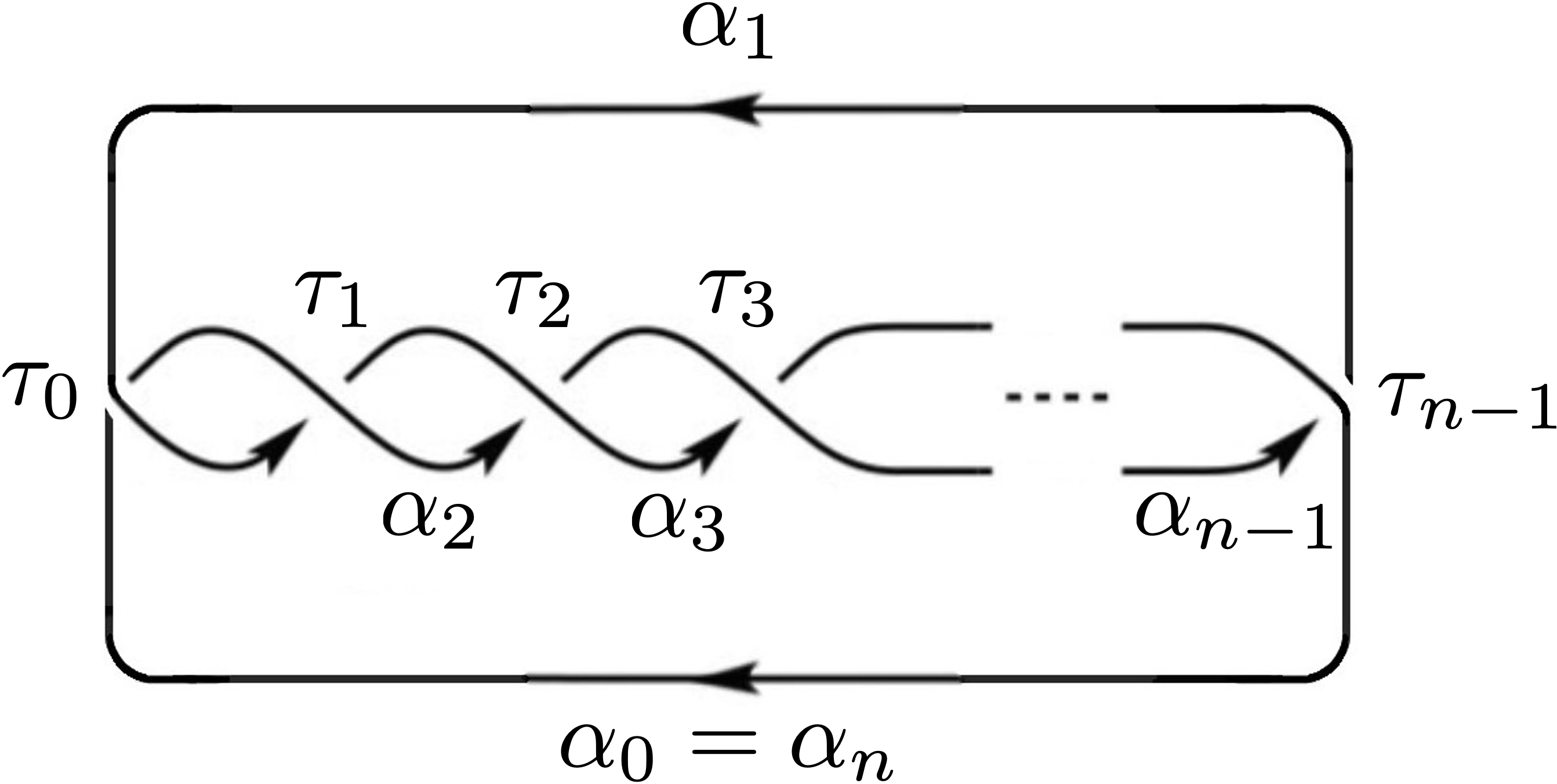}
  \caption{The diagram of $(2,n)$-torus knots and its arcs $\alpha_0,\cdots,\alpha_{n-1}$.}
  \label{pic_torus_knot_withcaption}
\end{figure}

\section{$S^2_1(r)$-Colorings of $(2,n)$-torus knots}
\label{section_determine_S^2_1(r)-coloring}
Suppose $r>0$ and $D$ be the diagram of the $(2,n)$-torus knot as shown in Fig. \ref{pic_torus_knot_withcaption}. We determine $S^2_1(r)$-colorings of $D$.

For an inner automorphism $\rho\in \operatorname{Inn}SL(2,\mathbb{R})$ induces a quandle automorphism of $S^2_1(r)$. We identify $\rho$ and the induced quandle automorphism in this section. The fact induces the action of $\operatorname{Inn}SL(2,\mathbb{R})$ on $\operatorname{Col}_{S^2_1(r)}(D)$.

By Proposition \ref{prop_general_ver_coloring}, we find pairs $x,y\in S^2_1(r)$ satisfying
\begin{equation}
\label{equation_determine_coloring}
    (xy)^{k}x=y(xy)^{k},
\end{equation}
where $k=\tfrac{n-1}{2}$, to determine the $S^2_1(r)$-colorings. Considering the action of $\operatorname{Inn}SL(2,\mathbb{R})$, it does not lose its generality as
\[
x
=
D(r)
=
\begin{pmatrix}
e^{r} & 0\\
0 & e^{-r}
\end{pmatrix}
,\ {}
y
=
\begin{pmatrix}
a & b\\
c & d
\end{pmatrix}.
\]
However, by Proposition \ref{prop_tr_expS2r}, let $a$, $b$, $c$, and $d$ be real numbers that satisfy the following conditions:
$ad-bc=1$ and $a+d=\cosh{r}$.

Let $\lambda$, $\mu\in\mathbb{C}$ be the two solutions of the characteristic equation of $y$
\[
t^2-(ae^{r}+de^{-r})t+1=0.
\]
We are able to determine $x,y$ satisfying equation (\ref{equation_determine_coloring}) in light of the argument in section \ref{section_equation_determine_coloring}.

\noindent(In the case $b=0$ or $c=0$) In light of lemma \ref{lem_presentation_lambda+mu_as_cos} and Lemma \ref{lem_case_b=0_or_c=0}, the real numbers $a$, $b$, $c$, and $d$ satisfying equation (\ref{equation_determine_coloring}) are
\[
(a,b,c,d)=(e^{r},0,0,e^{-r}).
\]
Then, $x$ and $y$ induces a trivial coloring in the sense of Example \ref{example_trivial-hom}. We denote the $S^2_1$-coloring as $C_0$.

\noindent(In the case $b\neq0$ and $c\neq0$ and $\lambda\neq\mu$) In light of Lemma \ref{lem_bneq0_cneq0_lneqm}, the real numbers $a$, $b$, $c$, and $d$ satisfying equation (\ref{equation_determine_coloring}) satisfy
\begin{eqnarray*}
(a,d)&=&
\left(
\frac{-e^{-r}\cosh{r}+\cos{2\theta_j}}{\sinh{r}},
\frac{e^{r}\cosh{r}-\cos{2\theta_j}}{\sinh{r}}
\right),\\
bc&=&-\frac{4\sin^2{\theta_j}(\sin^2{\theta_j}+\sinh^2{r})}{\sinh^2{r}},
\end{eqnarray*}
where $\theta_{j}=\tfrac{\pi j}{2n}$ and $j=1,3,\cdots,\frac{n-1}{2}$. We denote the $S^2_1(r)$-coloring induced by $x,y$ as $C_{j,b,c}$. 

\noindent(In the case $b\neq0$ and $c\neq0$ and $\lambda=\mu$)
 There is no $a$, $b$, $c$, or $d$ satisfying equation in (\ref{equation_determine_coloring}) light of Lemma  \ref{lem_bneq0_cneq0_l=m}.

We summarize the discussion in this section like the following theorem.
\begin{theorem}
Suppose $D$ be the diagram of $(2,n)$-torus knots as shown in Fig. \ref{pic_torus_knot_withcaption}. Then,
\begin{eqnarray*}
\operatorname{Col}_{S^2_{1}(r)}(D)=
\bigcup_{\rho\in\operatorname{Inn}SL(2,\mathbb{R})}
\left\{
\rho\circ C_0
\right\}
\cup
\left\{
\rho\circ C_{j,b,c}
\ :\ {}
\begin{matrix}
j=1,3,\cdots,n-2,\ b,c\in\mathbb{R}\\
bc=-\frac{4\sin^2{\theta_j}(\sin^2{\theta_j}+\sinh^2{r})}{\sinh^2{r}}
\end{matrix}
\right\}.
\end{eqnarray*}
\end{theorem}

Finally, we end this section by preparing a lemma used in section \ref{section_computing_longitudinal_map}.

\begin{lemma}
\label{lem_(u0u1)^n=-E}
Suppose $\alpha_0$, $\alpha_1$ are arcs of the diagram of $(2,n)$-torus knots as shown in Fig. \ref{pic_torus_knot_withcaption}. Then
\[
(C_{j,b,c}(\alpha_0)C_{j,b,c}(\alpha_{1}))^n
=
\begin{pmatrix}
-1 & 0\\
0 & -1
\end{pmatrix}.
\]
\end{lemma}
\begin{proof}
By Lemma \ref{lem_to_the_power_of_xy},
\begin{eqnarray*}
&&(C_{j,b,c}(\alpha_0)C_{j,b,c}(\alpha_{1}))^n=\\
&&\frac{1}{\lambda-\mu}
\begin{pmatrix}
-\lambda^{n-1}+\mu^{n-1}+ae^r(\lambda^n-\mu^n) & be^r(\lambda^n-\mu^n)\\
ce^{-r}(\lambda^n-\mu^n) & \lambda^{n+1}-\mu^{n+1}-ae^r(\lambda^n-\mu^n)
\end{pmatrix}.
\end{eqnarray*}
By Lemma \ref{lem_presentation_of_lambda_as_exp},
\[
\frac{\lambda^n-\mu^n}{\lambda-\mu}
=
\frac{\sin{n\theta_{j}}}{\sin{\theta_{j}}}
=
0
\]
and
\[
\frac{\lambda^{n\pm1}-\mu^{n\pm1}}{\lambda-\mu}
=
\frac{\sin{(n\pm1)\theta_{j}}}{\sin{\theta_{j}}}
=
\mp1.
\]
Thus the result follows.
\end{proof}

\section{Quandle colorings and representation of knot groups}
\label{section_introduction_to_NSK_1to1correspondence}
We introduce Nosaka's work to see the one-to-one correspondence between a quandle coloring and a representation of knot groups. See Nosaka \cite{Nosaka2015,Nosaka2017book} for more details. 

Let $K$ be a tame knot in the 3-sphere $S^3$ with a diagram $D$, $\pi_1(S^3\setminus K)$ the knot group of $K$ and $Q_K$ the knot quandle. For any augmented quandle $(X,G,\kappa)$, we define a set
\[
R(K,G)=\{
f\in\operatorname{Hom}(\pi_1(S^3\setminus K),G)
\ :\ {}
{}^{\exists}x\in X,\ f(\mathfrak{m})=\kappa(x)
\}.
\]
\begin{theorem}[Nosaka \cite{Nosaka2015}]\label{thm_Nosaka_correspondence}
Let $(X,G,\kappa)$ be a faithful augmented quandle. Then, there is a bijection
$
\Psi:\operatorname{Col}_X(D)
\xrightarrow{\sim}
R(K,G).
$
\end{theorem}
The bijection $\Psi$ is given as follows. It is known that $\pi_1(S^3\setminus K)$ is generated by the elements corresponding to the arcs of $D$ (Wirtinger presentation, see \cite{BurdeZieschang1985}). For any $X$-coloring $C$, $\Psi C:\pi_1(S^3\setminus K)\to G$ is a group homomorphism satisfying {}
$
\Psi f(\alpha)=\kappa\circ f(\alpha)
$ {}
for any $\alpha$ of $\pi_1(S^3\setminus K)$ corresponding to an arc of $D$.

We give a few facts about the bijection $\Psi$.
\begin{lemma}\label{lem_Nosaka-1to1_correspondence}
\begin{enumerate}
    \item The action of $G$ on $X$ induces the action of $G$ on $\operatorname{Col}_X(D)$. $X$-colorings $C_1,C_2:Q_K\to X$ are in the same $G$-orbit if and only if $\Psi C_1$ and $\Psi C_2$ are conjugate.
    \item An $X$-coloring $C$ is trivial in the sense of Definition \ref{def_trivial-hom} if and only if $\Psi C$ is an abelian representation.
\end{enumerate}
\end{lemma}

\begin{cor}
\label{cor_domain_longitudinal_mapping_colorings}
Suppose $\alpha_0$ is an arc of $D$, $x\in X$, and $\mathfrak{m}$ is equal to $\alpha_0$ as an element of the knot group $\pi_1(S^3\setminus K)$. Then the bijection $\Psi$ induces the one-to-one correspondence 
\begin{eqnarray*}
&&\{
f\in\operatorname{Hom}(\pi_1(S^3\setminus K),G)
\ :\ {}
f(\mathfrak{m})=\kappa(x)
\}
\simeq
\{
C\in
\operatorname{Col}_{X}(D)
\ :\ {}
C(\alpha_0)=x
\}.
\end{eqnarray*}
\end{cor}

\section{$S^2_1(r)$-colorings and hyperbolic representation of knot groups}
\label{section_S^2_1(r)-coloring_vs_hyperbolic_rep}
Suppose $D$ is the diagram of $(2,n)$-torus knots as shown in Fig. \ref{pic_torus_knot_withcaption}.

\begin{lemma}
\label{lem_hyperbolicrep_tr=2coshr}
Let $K$ be a knot and $\mathfrak{m}\in\pi_1(S^3\setminus K)$ be a meridian. Then,
\begin{eqnarray*}
&&\{
f\in\operatorname{Hom}(\pi_1(S^3\setminus K),SL(2,\mathbb{R}))
\ :\ {}
{}^{\exists}x\in S^2_1(r),\ f(\mathfrak{m})=x
\}\\
&=&
\{
f\in\operatorname{Hom}(\pi_1(S^3\setminus K),SL(2,\mathbb{R}))
\ :\ {}
\operatorname{trace}f(\mathfrak{m})=2\cosh{r}
\}.
\end{eqnarray*}
\end{lemma}
\begin{proof}
By Proposition \ref{prop_tr_expS2r}, this lemma is true.
\end{proof}

\begin{prop}
Suppose $K$ be a knot with a diagram $D$, and $r>0$. Then, there is a bijection
\[
\Psi_{K,r}:
\operatorname{Col}_{S_1^2(r)}(D)
\xrightarrow{\sim}
\{
f\in\operatorname{Hom}(\pi_1(S^3\setminus K),SL(2,\mathbb{R}))
\ :\ {}
\operatorname{trace}f(\mathfrak{m})=2\cosh r
\}
\]
\end{prop}
\begin{proof}
The result follows from Example \ref{example_presentation_as_augmented_quandle},  Theorem \ref{thm_Nosaka_correspondence}, and Lemma \ref{lem_hyperbolicrep_tr=2coshr}.
\end{proof}
By Lemma \ref{lem_Nosaka-1to1_correspondence}, we have the following properties.
\begin{enumerate}
    \item The action of $SL(2,\mathbb{R})$ on $S^2_1(r)$ induces the action of $SL(2,\mathbb{R})$ on $\operatorname{Col}_{S^2_1(r)}(D)$. Then, $S^2_1(r)$-colorings $C_1,C_2$ are in the same $SL(2,\mathbb{R})$-orbit if and only if $\Phi_{K, r}C_1$ and $\Phi_{K, r} C_2$ are conjugate.
    \item A $S^2_1(r)$-coloring $C$ is trivial in the sense of Definition \ref{def_trivial-hom} if and only if $\Phi_{K, r}C$ is abelian.
\end{enumerate}

Finally, we end this section by preparing a proposition used in section \ref{section_computing_longitudinal_map}.
\begin{prop}
\label{prop_relation_longitudinalmap_coloring}
Let $K$ be $(2,n)$-torus knots with a diagram $D$ as shown in Fig. \ref{pic_torus_knot_withcaption}, $\mathfrak{m}\in\pi_1(S^3\setminus K)$ be a meridian, and $D(r)$ be an $S^2_1(r)$-element defined in Proposition \ref{prop_injectivity_exp}. For $\rho\in\operatorname{Inn}SL(2,\mathbb{R})$, 
\begin{eqnarray*}
&&\{
f\in\operatorname{Hom}(\pi_1(S^3\setminus K),SL(2,\mathbb{R}))
\ :\ {}
f(\mathfrak{m})=\rho(D(r))
\}\\
&\simeq&
\{
C\in
\operatorname{Col}_{S^2_1(r)}(D)
\ :\ {}
C(\alpha_0)=\rho(D(r))
\}\\
&=&
\{\rho\circ C_0\}
\cup
\left\{
\rho\circ C_{j,b,c}
\ :\ {}
\begin{matrix}
j=1,3,\cdots,n-2,\ b,c\in\mathbb{R}\\
bc=-\frac{4\sin^2{\theta_j}(\sin^2{\theta_j}+\sinh^2{r})}{\sinh^2{r}}
\end{matrix}
\right\}
\end{eqnarray*}
\end{prop}
\begin{proof}
The result follows in light of the argument in section \ref{section_determine_coloring} and Corollorary \ref{cor_domain_longitudinal_mapping_colorings}.
\end{proof}

\section{The longitudinal mapping knot invariant for $SL(2,\mathbb{R})$}
\label{section_computing_longitudinal_map}
In this section, we determine the value of the longitudinal mapping for $SL(2,\mathbb{R})$ in the case $x\in S^2_1(r)$ and $K$ are $(2,n)$-torus knots.

We introduce the definition of the longitudinal mapping knot invariant defined by Clark-Saito \cite{ClarkSaito}. According to Clark-Saito \cite{ClarkSaito}, the longitudinal mapping is an extention of the quandle cocycle invariant defined by Carter et al. \cite{CJKLS}  and the knot colouring polynomial defined by Eisermann \cite{Eisermann2007}.
\begin{definition}[Clark-Saito \cite{ClarkSaito}]
Let $G$ be a group, $x$ be an elemnt of $G$, and $K$ be a knot. The longitudinal mapping is a map
\[
\mathcal{L}_{G}^{x}(K):
\{
f\in
\operatorname{Hom}(\pi_1(S^3\setminus K),G)
\ :\ {}
f(\mathfrak{m})=x
\}
\to G
\quad
f\mapsto f(\mathfrak{l}),
\]
where $\mathfrak{m}\in\pi_1(S^3\setminus K)$ is a meridian and $\mathfrak{l}\in\pi_1(S^3\setminus K)$ be a longitude. When there is no choice of confusion, we write $\mathcal{L}_{G}^{x}$ inplace of $\mathcal{L}_{G}^{x}(K)$.
\end{definition}

\begin{rem}
The definition of longitudinal mapping does not depend on the meridian-longitude pair. See \cite[Remark 3.3 and Theorem B.4]{ClarkSaito} for more detail.
\end{rem}

By Proposition \ref{prop_relation_longitudinalmap_coloring}, we identify the domain of the longitudinal mapping as a set of some quandle colorings.

Suppose $K$ be $(2,n)$-torus knot with a diagram $D$ as showed in Fig. \ref{pic_torus_knot_withcaption}.

\begin{theorem}
For $j=1,3,\cdots,n-2$ and an inner automorphism $\rho\in\operatorname{Inn}G$, 
\begin{eqnarray*}
\mathcal{L}^x_G(\Psi_{K,r}(\rho\circ C_0))
&=&
\begin{pmatrix}
1 & 0\\
0 & 1
\end{pmatrix},\\
\mathcal{L}^x_G(\Psi_{K,r}(\rho\circ C_{j,b,c}))
&=&
\rho(
\begin{pmatrix}
-e^{-2nr} & 0\\
0 & -e^{2nr}
\end{pmatrix}).
\end{eqnarray*}
\end{theorem}

\begin{proof}
By Lemma \ref{lem_correction_ClarkSaito_lem6.3}, for any representation $f:\pi_1(S^3\setminus K)\to SL(2,\mathbb{R})$,
\[
f(\mathfrak{l})=f(\alpha_{0})^{-2n}(f(\alpha_{0})f(\alpha_{1}))^{n}.
\]
By Proposition \ref{prop_relation_longitudinalmap_coloring}, we should consider the following two cases.

\noindent{(In the case $f=\Psi_{K,r}(C_{0})$)} Since  $f(\alpha_{0})=f(\alpha_{1})$,
\[
\mathcal{L}_{G}^{x}(f)
=
f(\alpha_{0})^{-2n}(f(\alpha_{0})f(\alpha_{0}))^{n}
=
\begin{pmatrix}
1 & 0\\
0 & 1
\end{pmatrix}.
\]

\noindent{(In the case $f=\Psi_{K,r}(C_{0})$)} By Lemma \ref{lem_(u0u1)^n=-E},
\[
\mathcal{L}_{G}^{x}(f)
=
\rho(
\begin{pmatrix}
e^{r} & 0\\
0 & e^{-r}
\end{pmatrix}^{-2n}
\begin{pmatrix}
-1 & 0\\
0 & -1
\end{pmatrix})
=
\rho(
\begin{pmatrix}
-e^{-2nr} & 0\\
0 & -e^{2nr}
\end{pmatrix}).
\]
\end{proof}

\appendix
\section{A presentation of a longitude of $(2,n)$-torus knots}
We see a presentation of a longitude of $(2,n)$-torus knots. The content of this section has already been done in \cite[Lemma 6.3]{ClarkSaito} essentially, but there is a fatal typographical error in the proof, so we prove it again for completeness. See Remark \ref{rem_ClarkSaito_Lem6.3} for more details on \cite[Lemma 6.3]{ClarkSaito}.

Let $K$ be $(2,n)$-torus knots with a diagram $D$ as showed in Fig. \ref{pic_torus_knot_withcaption}. The knot group $\pi_{1}(S^3\setminus K)$ has a Wirtinger presentation with respect to $D$:
\[
\left\langle
\alpha_{0},\cdots,\alpha_{n-1}
\ :\ {}
\alpha_{i+2}^{-1}\alpha_{i+1}^{-1}\alpha_{i}\alpha_{i+1} 
\quad
(j = 0,1,\cdots,{n-2})
\right\rangle.
\]

\begin{lemma}
\label{lem_presentation_of_arcs_wrt_Wirtinger_presentation}
For $j=0,\cdots,\tfrac{n-1}{2}$,
\begin{eqnarray*}
\alpha_{2j}&=&(\alpha_{0}\alpha_{1})^{-j}\alpha_{0}(\alpha_{0}\alpha_{1})^{j},\\
\alpha_{2j+1}&=&(\alpha_{0}\alpha_{1})^{-j}\alpha_{1}(\alpha_{0}\alpha_{1})^{j}.
\end{eqnarray*}
\end{lemma}
\begin{proof}
The result follows by induction on $j$.
\end{proof}

\begin{lemma}[c.f. Clark-Saito \cite{ClarkSaito}]
\label{lem_correction_ClarkSaito_lem6.3}
A longitude $\mathfrak{l}\in\pi_1(S^3\setminus K)$ has the following presentation:
\[
\mathfrak{l}=\alpha_{0}^{-2n}(\alpha_{0}\alpha_{1})^n
\]
\end{lemma}
\begin{proof}
Suppose $n=2k+1$. By \cite[3.13 Remark.]{BurdeZieschang1985},
\[
\mathfrak{l}
=
(\alpha_{1}\alpha_{3}\cdots\alpha_{2k-1})(\alpha_{0}\alpha_{2}\cdots\alpha_{2k})\alpha_{0}^{-n}.
\]
By Lemma \ref{lem_presentation_of_arcs_wrt_Wirtinger_presentation},
\begin{eqnarray*}
\alpha_{1}\alpha_{3}\cdots\alpha_{2k-1}
&=&
\alpha_{0}^{-k}(\alpha_{0}\alpha_{1})^{k},\\
\alpha_{0}\alpha_{2}\cdots\alpha_{2k}
&=&
\alpha_{0}\alpha_{1}^{-k}(\alpha_{0}\alpha_{1})^{k}.
\end{eqnarray*}
Therefore, by Lemma \ref{lem_presentation_of_arcs_wrt_Wirtinger_presentation} and $\alpha_{0}=\alpha_{n}=\alpha_{2k+1}$,
\[
\mathfrak{l}
=\alpha_{0}^{-k}(\alpha_{0}\alpha_{1})^{k+1}\alpha_{1}^{-k-1}(\alpha_{0}\alpha_{1})^{k}\alpha_{0}^{-n}
=\alpha_{0}^{-k}(\alpha_{0}\alpha_{1})^{n}\alpha_{0}^{-n-k-1}.
\]
The result follows since meridians and longitudes commute (see \cite[3.14 Definition and Proposition]{BurdeZieschang1985}).
\end{proof}

\begin{rem}
\label{rem_ClarkSaito_Lem6.3}
The statement of \cite[Lemma 6.3]{ClarkSaito} is true. However, the proof of \cite[Lemma 6.3]{ClarkSaito} contains a fatal typographical error: Clark and Saito presented a image of a longitude of $(2,n)$-torus knots as 
\begin{equation*}
\label{typo_ClarkSaito}
\mathcal{L}(C)
=
q_{0}^{-n}(q_{1}q_{3}\cdots q_{2k-3})(q_{0}q_{2}\cdots q_{2k})
\end{equation*}
to prove \cite[Lemma 6.3]{ClarkSaito}, but the correct presentation is 
\[
\mathcal{L}(C)
=
q_{0}^{-n}(q_{1}q_{3}\cdots q_{2k-1})(q_{0}q_{2}\cdots q_{2k}).
\]
\end{rem}

\section{The solutions of a equation $\lambda^m-\mu^m=\lambda^{m+1}-\mu^{m+1}$}
Suppose $m$ is a positive integer. We see properties of $\lambda,\mu\in\mathbb{C}$ satisfying following conditions:
$
\lambda\mu=1
$, {}
$
\lambda\neq\mu
$, and {}
$
\lambda^m-\mu^m=\lambda^{m+1}-\mu^{m+1}
$.
\begin{lemma}
\label{lem_presentation_of_lambda_as_exp}
\[
\lambda\in
\left\{
\operatorname{exp}\frac{\pi j\sqrt{-1}}{2m+1}
\ :\ {}
j=1,3,\cdots,2m-1,2m+3,2m+5,\cdots,4m+1
\right\}.
\]
\end{lemma}
\begin{proof}
Since $\lambda$ and $\mu$ satisfy $\lambda\mu=1$,
\begin{eqnarray*}
0&=&
\lambda^m-\mu^m-\lambda^{m+1}+\mu^{m+1}\\
&=&
\sum_{i=0}^{m}\lambda^{2i-m}
-
\sum_{i=1}^{m}\lambda^{2i-m-1}\\
&=&
\frac{\lambda^{-m}(\lambda^{2m+1}+1)}{\lambda+1}
.
\end{eqnarray*}
Considering $(\lambda,\mu)\neq(\pm1,\pm1)$, the result follows.
\end{proof}
We get the following lemma in light of Lemma \ref{lem_presentation_of_lambda_as_exp}.
\begin{lemma}
\label{lem_presentation_lambda+mu_as_cos}
\[
\lambda+\mu\in
\left\{
2\cos\frac{\pi j}{2m+1}
\ :\ {}
j=1,3,\cdots,2m-1
\right\}.
\]
\end{lemma}

\section{The equation \ref{equation_determine_coloring}}
\label{section_equation_determine_coloring}
Suppose $k$ is a positive integer, $r>0$ and
\[
x
=
D(r)
=
\begin{pmatrix}
e^{r} & 0\\
0 & e^{-r}
\end{pmatrix}
,\ {}
y
=
\begin{pmatrix}
a & b\\
c & d
\end{pmatrix}.
\]
We determine the real numbers $a$, $b$, $c$, and $d$ satisfying following conditions: $a+d=2\cosh{r}$, $ad-bc-1$, and the equation \ref{equation_determine_coloring}, that is,
\[
(xy)^{k}x=y(xy)^{k}.
\]
Let $\lambda$, $\mu\in\mathbb{C}$ be the two solutions of the characteristic equation of $y$
\[
t^2-(ae^{r}+de^{-r})t+1=0.
\]

\begin{lemma}
\label{lem_to_the_power_of_xy}
For a positive integer $m$,
\[
(xy)^m
=
\left\{
\begin{array}{ccl}
\begin{pmatrix}
a^me^{rm} & 0\\
ce^{-r}\sum_{i=0}^{m-1}(ae^{r})^i(de^{-r})^{m-i-1} & d^me^{-rm}
\end{pmatrix}
    &\mbox{if}& b=0,  \\
\begin{pmatrix}
a^{m}e^{rm} & be^r\sum_{i=0}^{m-1}(ae^r)^i(de^{-r})^{m-i-1}\\
0 & d^me^{-rm}
\end{pmatrix}
     &\mbox{if}& c=0, \\
M_1(m)
&\mbox{if}& b\neq0\mbox{ and }c\neq0\mbox{ and }\lambda\neq\mu,\\
M_2(m)
&\mbox{if}& b\neq0\mbox{ and }c\neq0\mbox{ and }\lambda=\mu,
\end{array}
\right.
\]
where $M_1(m)$ is a matrix
\[
\frac{1}{\lambda-\mu}
\begin{pmatrix}
-\lambda^{m-1}+\mu^{m-1}+ae^r(\lambda^m-\mu^m) & be^r(\lambda^m-\mu^m)\\
ce^{-r}(\lambda^m-\mu^m) & \lambda^{m+1}-\mu^{m+1}-ae^r(\lambda^m-\mu^m)
\end{pmatrix}
\]
and $M_2(m)$ is a matrix
\[
\begin{pmatrix}
\left(-(m+1)\lambda+ame^r\right)\lambda^{m-1} & be^{r}m\lambda^{m-1}\\
ce^{-r}m\lambda^{m-1} & -\lambda^{m-1}\left((m-1)\lambda-ame^r\right)
\end{pmatrix}.
\]
\end{lemma}
\begin{proof}
\noindent(In the case $b=0$ or $c=0$) The result follows by the induction on $m$.

\noindent(In the case $b\neq0$ and $c\neq0$ and $\lambda\neq\mu$) The matrix $xy$ is diagonalizable as follows:
\[
xy=
\begin{pmatrix}
be^r & be^r\\
\lambda-ae^r & \mu -ae^r
\end{pmatrix}
\begin{pmatrix}
\lambda & 0\\
0 & \mu
\end{pmatrix}
\begin{pmatrix}
be^r & be^r\\
\lambda-ae^r & \mu -ae^r
\end{pmatrix}^{-1}.
\]
Therefore, the result follows by direct computation.

\noindent(In the case $b\neq0$ and $c\neq0$ and $\lambda\neq\mu$) The matrix $xy$ has the Jordan normal form
\[
xy=
\begin{pmatrix}
be^r & 0\\
\lambda-ae^r & 1
\end{pmatrix}
\begin{pmatrix}
\lambda & 1\\
0 &\lambda
\end{pmatrix}
\begin{pmatrix}
be^r & 0\\
\lambda-ae^r & 1
\end{pmatrix}^{-1}.
\]
Thus the result follows by direct computation.
\end{proof}

The following lemmas are derived from Lemma \ref{lem_to_the_power_of_xy}.
\begin{lemma}
\label{lem_case_b=0_or_c=0}
$b=0$ or $c=0$ if and only if $(a,b,c,d)=(e^r,0,0,e^{-r})$.
\end{lemma}

\begin{lemma}
\label{lem_bneq0_cneq0_lneqm}
$b\neq0$ and $c\neq0$ and $\lambda\neq\mu$ if and only if 
$
\lambda^k-\mu^k=\lambda^{k+1}-\mu^{k+1}
$.
\end{lemma}
\begin{lemma}
\label{lem_bneq0_cneq0_l=m}
If $b\neq0$ and $c\neq0$ and $\lambda=\mu$, there is no $a$, $b$, $c$ or $d$ satisfying the conditions.
\end{lemma}

\section*{Acknowledgement}
The author is grateful to Professor Hiroyuki Ochiai, Kyushu University, for many valuable comments and discussions. He also thanks Michiko Yonemura, University of Miyazaki, for drawing Fig. \ref{pic_torus_knot_withcaption} for him.

This work was supported by JST SPRING, Grant Number JPMJSP2136.

\textit{Department of Mathematics, Kyushu University, 744 Motooka, Nishi-ku, Fukuoka 819–0395, Japan}
\end{document}